\documentclass[11pt, reqno, a4paper, final]{amsart}

\usepackage[latin1]{inputenc}
\usepackage[english]{babel}
\usepackage[T1]{fontenc}
\usepackage{amsmath,amsfonts,amssymb,amsthm,amscd,enumitem}
\usepackage{eucal} 
\usepackage{color}
\usepackage{hyperref}
\usepackage{a4wide}										 	
\usepackage{mathrsfs}									  
\usepackage[all]{xy}										
\usepackage{url}												
\usepackage[notref,notcite,draft]{showkeys}   
\usepackage{xfrac} 

\theoremstyle{plain}

\newtheorem{thm}{Theorem}[section]

\newtheorem{lem}[thm]{Lemma}
\newtheorem{lemma}[thm]{Lemma}
\newtheorem{prop-defi}[thm]{Definition \& Proposition}
\newtheorem{prop}[thm]{Proposition}

\newtheorem*{thm*}{Theorem}
\newtheorem*{prop*}{Proposition}
\newtheorem*{cor*}{Corollary}

\theoremstyle{definition}
\newtheorem{defi}[thm]{Definition}
\newtheorem{definition}[thm]{Definition}

\newtheorem{rem}[thm]{Remark}

\newtheorem*{added}{Note added in proof}
\newcounter{claim-counter}
\newtheorem{claim}[claim-counter]{Claim}
\newtheorem*{claim*}{Claim}

\renewcommand{\AA}{{\mathbb A}}
\newcommand{\NN}{{\mathbb N}}
\newcommand{\ZZ}{{\mathbb Z}}
\newcommand{\BB}{{\mathbb B}}

\newcommand{\CC}{{\mathbb C}}
\newcommand{\KK}{{\mathbb K}}

\newcommand{\MM}{{\mathbb M}}
\newcommand{\GG}{{\mathbb G}}
\newcommand{\HH}{{\mathbb H}}
\newcommand{\FF}{{\mathbb F}}

\newcommand{\Z}{{\mathcal Z}}

\newcommand{\B}{{\mathscr B}}

\newcommand{\X}{{\mathcal X}}

\renewcommand{\max}{{\operatorname{max}}}

\newcommand{\varps}{{\varepsilon}}

\newcommand{\tens}{\otimes}

\newcommand{\spann}{{\operatorname{span}}}

\newcommand{\red}{{\operatorname{red}}}
\newcommand{\Tor}{\operatorname{Tor}}
\newcommand{\Hom}{\operatorname{Hom}}
\newcommand{\id}{\operatorname{id}}

\newcommand{\Alg}{{\operatorname{Alg}}}
\newcommand{\Pol}{{\operatorname{Pol}}}
\newcommand{\Ext}{{\operatorname{Ext}}}
\renewcommand{\leq}{\leqslant}
\renewcommand{\geq}{\geqslant}
\newcommand{\Cohom}{\operatorname{H}}

\usepackage{thmtools}
\declaretheorem[style=theorem,name={Theorem}]{theoremletter}

\renewcommand{\L}{\operatorname{\ell}}

\title{On the $\L^2$-Betti numbers of universal quantum groups}
\author{David Kyed}
\address{David Kyed, Department of Mathematics and Computer Science, University of Southern Denmark, Campusvej 55, DK-5230 Odense M, Denmark}
\email{dkyed@imada.sdu.dk}

\author{Sven Raum}
\address{Sven Raum,
EPFL SB SMA,
Station 8,
CH-1015 Lausanne,
Switzerland}
\email{sven.raum@epfl.ch}

\thanks{D.K.~gratefully acknowledges the financial support from  the Villum foundation (grant 7423). Sven Raum's research leading to these results has received funding from the People Programme (Marie Curie Actions)
of the European Union's Seventh Framework Programme (FP7/2007-2013) under REA grant agreement n$^\circ$[622322].}

\begin{document}
\begin{abstract} 
We show that the first $\L^2$-Betti number of the duals of the free unitary quantum groups is one,  and that all $\L^2$-Betti numbers vanish for the duals of the quantum automorphism groups of full matrix algebras.

\end{abstract}

\maketitle

\section*{Introduction}
 
A discrete quantum group is the natural replacement for a discrete group in the setting of non-commutative geometry, where one replaces spaces and varieties by suitable algebras or categories of functions and then drops the commutativity assumptions on these.  One approach to discrete quantum groups, formulated in an operator algebraic language by Woronowicz \cite{woronowicz,wor-cp-qgrps},  fits naturally into the more general framework of locally compact quantum groups developed by Kustermans and Vaes \cite{kustermans-vaes-C*-lc, kustermans-vaes, kustermans-universal}.  Thanks to this operator algebraic formulation, numerous  aspects of analytic group theory have been successfully and fruitfully extended to the setting of discrete quantum groups (cf.~\cite{woronowicz, murphy-tuset, brannan-approximation, vergnioux-rapid-decay, fima-prop-T, vaes06, meyer-nest, voigt-bc-for-free-orthogonal}), including  the notion of $\L^2$-Betti numbers, which was introduced for discrete quantum groups in \cite{quantum-betti} and  is the main concern of the present article.  While for ordinary discrete groups, computational results regarding their $\L^2$-Betti numbers are ample, for quantum groups the situation is quite different: beyond the case of amenable discrete quantum groups, for which all $\L^2$-Betti numbers vanish \cite{coamenable-betti}, and the somewhat artificial examples constructed in \cite{kunneth-formula}, the work of  Vergnioux  \cite{vergnioux-paths-in-cayley}  and Collins-H{\"a}rtel-Thom \cite{thom-collins} provides the only computation  of $\L^2$-Betti numbers for genuine quantum examples.  In \cite{vergnioux-paths-in-cayley}, Vergnioux used intricate arguments involving so-called quantum Cayley trees to show that the first $\L^2$-Betti number of the discrete dual $\hat O_n^+$ of the free orthogonal quantum groups $O_n^+$ vanishes.  Later Collins-H{\"a}rtel-Thom \cite{thom-collins} used computations with Gr{\"o}bner bases in order to provide an explicit resolution of the trivial $\hat O_n^+$-module, and combining this with Vergnioux's result they proved the vanishing of all $\L^2$-Betti numbers of $\hat O_n^+$. In \cite{vergnioux-paths-in-cayley}, Vergnioux also proved that the first $\L^2$-Betti number of the discrete dual $\hat U_n^+$ of the free unitary quantum group $U_n^+$ is non-zero, but could not provide a precise calculation.  He conjectured, however, that $\beta_1^{(2)}(\hat U_n^+) = 1$ holds for all $n \geq 2$.  Our main theorem  verifies this conjecture and  thus provides the first computation of a non-zero $\L^2$-Betti number of a genuine quantum group.
\begin{theoremletter}\label{thm-A}
  For all $n\geq 2$ one has $\beta_{1}^{(2)}(\hat{U}_n^+)= 1$. 
\end{theoremletter}

As for discrete groups, to every discrete quantum group $\hat{\GG}$ one associates a natural von Neumann algebra $L\hat{\GG}$.  The relevance of Theorem \ref{thm-A} also stems from the close connection between the von Neumann algebras $L(\hat U_n^+)$ and the elusive free group factors $L\FF_n$. For $n=2$, Banica showed in \cite{banica-unitary} that $L(\hat U_2^+)$ is isomorphic to the free group factor $L\FF_2$, thereby providing a new  interesting model for the latter, and since then it has been an intriguing question to determine whether $L(\hat U_n^+)$ is a free group factor also for $n \geq 3$.  A large number of results comparing the analytic theory of $\hat U_n^+$ to that of free groups find strong similarities: the discrete  quantum groups $\hat{U}_n^+$ have rapid decay \cite{vergnioux-rapid-decay}, the Haagerup property \cite{brannan-approximation}, the Akemann-Ostrand property \cite{vergnioux-orientation-of-quantum-cayley-trees} and they give rise to  simple non-nuclear (reduced) $C^*$-algebras \cite{banica-unitary} and  full, prime, finite, factorial von Neumann algebras \cite{vaes-van-der-vennet} without Cartan subalgebras.  In stark contrast to this, our Theorem \ref{thm-A} demonstrates a behaviour  of $\hat U_n^+$ different from that of the free groups, in that $\beta_1^{(2)}(\FF_n)=n-1$ depends on the value of $n$ while this is not the case for $\beta_1^{(2)}(\hat{U}_n^+)$. Note that the isomorphism $L(\hat{U}_2^+) \cong L(\FF_2)$ is compatible with the intuition provided by our calculation, but that the independence on $n$ of the value of $\beta_1^{(2)}(\hat{U}_n^+)$ has no concrete bearing on the ability of $L\hat{U}_n^+$ to be a free group factor.\\

The proof of Theorem \ref{thm-A} takes Vergnioux's results in \cite{vergnioux-paths-in-cayley} as a main ingredient, thereby avoiding subtle considerations regarding quantum Cayley graphs. Instead, we carefully study extension properties of $1$-cocycles on quantum groups to find convenient representatives of the 1-cohomology classes of $\hat U_n^+$.  We then involve the duals of the quantum automorphism group $\AA_n$ of the matrix algebra $\MM_n(\CC)$, which appear as quantum subgroups of $\hat U_n^+$.  More precisely,  we need the following vanishing result for the first $\L^2$-Betti number of $\hat \AA_n$.
\begin{theoremletter}
\label{thm-B}
For all $n\geq 2$ and all $p\geq 0$ one has $\beta_p^{(2)}(\hat{\AA}_n)=0$.
\end{theoremletter}
 Note that the vanishing of $\beta_p^{(2)}(\hat{\AA}_n)$ for $p\geq 4$ also follows from \cite[Theorem 6.5]{bichon-gerstenhaber},  which shows that the cohomological dimension of $\hat \AA_n$ is equal to 3.\\

In addition to the introduction, the paper consists of two sections and an appendix.  The first of these sections contains  the relevant background material (including the definition of the objects mentioned above) and the second contains the proofs of our two main results.  In the appendix, we give a short proof of a well know ring-theoretical result in  an operator algebraic language.

\section*{Acknowledgement}

The authors thank Julien Bichon for pointing out the reference \cite{brown-goodearl}, which provides a reference for the result shown in the appendix.

\section{Preliminaries}

\subsection{Compact and discrete quantum groups}
The aim of this section is to fix our notation concerning quantum groups, but since this is by now fairly standard, these preliminaries will be kept rather brief. For more exhaustive details, we refer the reader to the original papers by Woronowicz \cite{woronowicz,wor-cp-qgrps} or the introductory texts \cite{tuset, timmermann-book}  as well as references therein.
In Woronowicz' approach to compact quantum groups, such an object --- here denoted  $\GG$ --- consists of a unital $C^*$-algebra $C(\GG)$ together with a $*$-homomorphism $\Delta \colon C(\GG) \to C(\GG)\tens_{\min} C(\GG)$ (the \emph{comultiplication}) satisfying a certain coassociativity- and non-degeneracy-condition.  Associated with this data is a distinguished state $h_\GG$ (the \emph{Haar state}) which plays the role of the the Haar integral of a compact group. The additional requirements on $\Delta$ ensure that if the $C^*$-algebra $C(\GG)$ happens to be commutative, then there exists a genuine compact group $G$ such that $C(\GG)=C(G)$ and $\Delta$ is dual to the multiplication map $G\times G\to G$.  Such a compact quantum group $\GG$  naturally gives rise to a dense Hopf $*$-subalgebra $\Pol(\GG)\subset C(\GG)$  and the representation on the GNS space $L^2(\GG)$ associated with $h_\GG$ restricts to an embedding $\Pol(\GG) \subset \BB(L^2(\GG))$. We shall write $C(\GG)_\red$ and $L^\infty(\GG)$ for the $C^*$-algebra and von Neumann algebra generated by $\Pol(\GG)$ inside $\BB(L^2(\GG))$, respectively. We will only be interested in the situation where $h_{\GG}$ is a trace (in which case $\GG$ is said to be of \emph{Kac type}); then  $L^\infty(\GG)$ is a finite von Neumann algebra and we may therefore consider the associated algebra $M(\GG)$  of closed, densely defined, unbounded operators affiliated with it.  Dual to the notion of a  unitary representation of a group is the notion of a \emph{unitary corepresentation} of a quantum group and if one specifies a finite dimensional such --- which is then a matrix over $\Pol(\GG)$ --- whose matrix coefficients generate $\Pol(\GG)$ then $\GG$ is  said to be a \emph{compact matrix quantum group}; this is the non-commutative analogue of a Lie group with fixed fundamental representation. \\

Associated with a compact quantum group $\GG$ is its so-called discrete dual quantum group $\hat{\GG}$, and it is fruitful to think of the associated algebras as being generalizations of the various group algebras  associated with a discrete group --- thus $\Pol(\GG)$ can be thought of as representing $\CC[\hat{\GG}]$, $C(\GG)_{\red}$ as representing $C^*_\red(\hat{\GG})$ and $L^\infty(\GG)$ as representing $L\hat{\GG}$. As one might expect, any countable discrete group $\Gamma$ does indeed give rise to a compact quantum group in this way --- its $C^*$-algebra being $C^*_{\red}\Gamma$ and the comultiplication being given by $\Delta(\lambda_\gamma)=\lambda_\gamma \tens\lambda_\gamma$.

\subsection{Universal quantum groups}
\label{sec:universal-quantum-groups}

In this section we introduce the quantum groups under consideration in the sequel.  

\begin{definition}[\cite{wang}]
The \emph{free unitary quantum group} $U_n^+$ is defined as the maximal $C^*$-completion of the universal, unital $*$-algebra $\Pol(U_n^+)$ generated by $n^2$ elements $\{u_{ij} \mid i,j=1,\dots, n\}$ subject to the relations making $u:=(u_{ij})_{i,j=1}^n$ and $\bar{u}:=(u_{ij}^*)_{i,j=1}^n$ unitary matrices. The comultiplication and counit are given on the generators by 
\begin{align}\label{delta-and-epsilon}
\Delta(u_{ij})=\sum_{k=1}^n u_{ik}\tens u_{kj} \quad \text{ and } \quad \varps(u_{ij})=\delta_{i,j},
\end{align}
and $u$ is a fundamental unitary corepresentation.
\end{definition}
\begin{definition}[\cite{wang}]
 The \emph{free orthogonal quantum group} $O_n^+$ is defined as the maximal $C^*$-completion of the universal, unital $*$-algebra $\Pol(O_n^+)$ generated by $n^2$ {selfadjoint} elements $\{v_{ij} \mid i,j=1,\dots, n\}$ subject to the relations making $v:=(v_{ij})_{i,j=1}^n$ an orthogonal matrix
The comultiplication and counit are given on the generators by the obvious analogue of \eqref{delta-and-epsilon} and $v$ is a fundamental unitary corepresentation.
\end{definition}
\begin{rem}
The reason for the names is twofold: firstly, if one additionally imposes the relation that the generators commute, then the resulting compact quantum groups identify with the classical orthogonal and unitary groups, respectively, and secondly the passage from classical groups to free quantum groups parallels the passage from classical to free probability in many respects.  For more information about $O_n^+$ and $U_n^+$ and their co-representation theory the reader is referred to \cite{banica-orto, banica-unitary, wang, wang-van-daele}.
\end{rem}

We will furthermore need the \emph{quantum automorphism group} $\AA_n$ of the matrix algebra $\MM_n(\CC)$. This quantum group is defined as the universal object in the category of quantum groups coacting (trace-preservingly) on $\MM_n(\CC)$, and can also be defined abstractly in terms of generators and relations, analogous to the case of $U_n^+$ and  $O_n^+$ above  \cite{wang-quantum-symmetry-groups}. For our purposes,  the following description is the relevant one:
\begin{lem}[{\cite[Corollary 4.1]{banica-sym}}]\label{A_n-description}
$\Pol(\AA_n)$ is isomorphic to the Hopf-subalgebra in $\Pol(O_n^+)$ generated by the elements $\{v_{ij}v_{kl}\mid i,j,k,l=1,\dots, n\}$.
\end{lem}
\begin{rem}
Note that the universal $C^*$-completion $C(\AA_n)_\max$ is the algebra commonly denoted $A_{\text{aut}}(\MM_n(\CC))$ in the literature \cite{wang-quantum-symmetry-groups, banica-sym}.\\
\end{rem}

Lastly, we denote by $\HH_n$ the compact quantum group corresponding to the Hopf $*$-algebra $\Pol(S^1)\ast \Pol(O_n^+)$ (see \cite{wang} for the free product construction of quantum groups) and recall \cite{banica-unitary} that $\Alg^*(zv_{ij}\mid i,j=1,\dots, n)\subset \Pol(\HH_n)$ is a Hopf $*$-subalgebra isomorphic to $\Pol(U_n^+)$;  the isomorphism being given by  $u_{ij}\mapsto zv_{ij}$, where $z$ denotes the identity map on the unit circle $S^1$. In what follows, we will always think of $\Pol(U_n^+)$ as  realized inside $\Pol(\HH_n)$ in this way.  Since $u_{ij}=zv_{ij}$ we have $u_{ij}^*u_{kl}=v_{ij}v_{kl}$ and by Lemma \ref{A_n-description}, $\Pol(\AA_n)$ is therefore also a subalgebra of $\Pol(U_n^{+})\subset \Pol(\HH)$.  The preceding discussion may be summarized by means of the following diagram of inclusions.
\begin{equation*}
  \xymatrix{
    \Pol(\AA_n) \ar@{^(->}[r] \ar@{^(->}[d] & \Pol(O_n^+) \ar@{^(->}[d]\\
    \Pol(U_n^+) \ar@{^(->}[r] & \Pol(\HH_n) \text{.}
  }
\end{equation*}

\subsection{$\L^2$-Betti numbers for quantum groups}
\label{sec:ltwo-betti-numbers}

  As indicated in the introduction, numerous notions from the theory discrete groups have been extended to the setting of discrete quantum groups and among these is the theory of $\L^2$-Betti numbers \cite{quantum-betti}. Following L{\"u}ck's approach \cite{Luck02}, the $\L^2$-Betti numbers
of the discrete dual of a compact quantum group $\GG$ of Kac type are defined as
\begin{equation*}
  \beta_p^{(2)}(\hat{\GG})
  :=
  \dim_{L^\infty(\GG)} \Tor_p^{\Pol(\GG)}(L^\infty(\GG),\CC)
  \text{,}
\end{equation*}
where $\dim_{L^\infty(\GG)}$ is L{\"u}ck's extended von Neumann dimension computed with respect to the trace $h_{\GG}$ (cf.~\cite{Luck02}).

\subsection{1-cohomology  for quantum groups}
\label{sec:one-cohomology}

In this section we describe how the first $\L^2$-Betti number of a discrete quantum group of Kac type can be obtained via 1-cohomology.  Let $\GG$ be a compact quantum group and denote by $\varps$ the counit on the associated Hopf $*$-algebra $\Pol(\GG)$.

\begin{defi}
  Let $\X$ be  a complex vector space with a representation of the ring $\Pol(\GG)$.
  \begin{enumerate}
  \item
    A \emph{1-cocycle} into $\X$ is a linear map $c\colon\Pol(\GG)\to \X$ satisfying $c(ab)=a.c(b) +c(a)\varps(b)$. The space of 1-cocycles is denoted $Z^1(\Pol(\GG), \X)$.
  \item
    A cocycle $c\in Z^1(\Pol(\GG), \X)$ is said to be \emph{inner} if there exists $\xi\in \X$ such that $c(a)=a.\xi-\varps(a)\xi$, and the space of inner cocycles is denoted $B^1(\Pol(\GG), \X)$.
  \item
    The \emph{first cohomology} $\Cohom^1(\Pol(\GG),\X)$ of $\Pol(\GG)$  with values in $\X$ is defined   as the space of cocycles modulo the inner ones. In analogy with the group case, we will also refer to this as the first cohomology of $\hat{\GG}$ with coefficients in $\X$.
  \item
    A vector $\xi \in \X$ is said to be \emph{fixed} if $x.\xi=\varps(x)\xi$ for all $x\in \Pol(\GG)$.
  \end{enumerate}
\end{defi}

\begin{lem}
  \label{uniquely-determined-lem}
  If $\GG$ is a compact matrix quantum group with  fundamental unitary corepresentation $u=(u_{ij})_{i,j=1}^n\in \MM_n(\Pol(\GG))$, then any cocycle $c$ into any left $\Pol(\GG)$-module $\X$
  is uniquely determined by its values on either of the sets $\{u_{ij}\mid i,j=1,\dots, n\}$ and $\{u_{ij}^*\mid i,j=1,\dots, n\}$.
\end{lem}
\begin{proof}
  First note that $\Pol(\GG) = \Alg(u_{ij}, u_{ij}^*)$ and due to the cocycle relation, $c$ is therefore uniquely determined by its values on the set $\{u_{ij}, u_{ij}^*\mid 1\leq i,j\leq n\}$.
  Writing the equation $1_n=u^*u=uu^*$ out in terms of matrix entries gives the relations:
  \begin{align}
    \sum_{k=1}^n u_{ki}^*u_{kj}&=\delta_{i,j}1 \label{unitary-eq-1},\\
    \sum_{k=1}^n u_{ik}u_{jk}^*&=\delta_{i,j}1\label{unitary-eq-2},
  \end{align}
  for all $i,j=1,\dots, n$. Since $c(1)=0$, the relation  \eqref{unitary-eq-1} gives that
  \[
  0=\sum_{k=1}^n c(u_{ki}^*u_{kj})=\sum_{k=1}^n u_{ki}^*c(u_{kj}) +\sum_{k=1}^n c(u_{ki}^*)\underbrace{\varps(u_{kj})}_{=\delta_{k,j}} =\sum_{k=1}^n u_{ki}^*c(u_{kj}) +c(u_{ji}^*).
  \]
  So, $c(u_{ji}^*)=-\sum_{k=1}^n u_{ki}^*c(u_{kj})$, so the values of $c$ on the set $\{u_{ij}\mid 1\leq i,j\leq n \}$ determine it on the set $\{u_{ij}^*\mid 1\leq i,j\leq n \}$ and thus, in turn, on all of $\Pol(\GG)$. Similarly, using \eqref{unitary-eq-2} one sees that the  values of $c$ on the set $\{u_{ij}^*\mid 1\leq i,j\leq n \}$ determine its values on $\{u_{ij}\mid 1\leq i,j\leq n \}$.
\end{proof}

The next lemma gives a precise link between first $\L^2$-Betti numbers and the the first cohomology group.  Its proof combines several well-known facts and we do not claim originality of the result.
\begin{lemma}
  \label{lem:ltwo-betti-numbers-and-cohomology}
  Let $\GG$, $\KK$ be compact quantum groups such that $\Pol(\GG) \subset \Pol(\KK)$ is a Hopf $*$-subalgebra.  Then
  \begin{enumerate}
  \item $\beta_1^{(2)}(\hat \GG) = \dim_{L^\infty(\KK)} \Cohom^1(\Pol(\GG), M(\KK))$.
  \item $\beta_1^{(2)}(\hat \GG)  = 0$ if and only if $\dim_{L^\infty(\KK)} \Cohom^1(\Pol(\GG), M(\KK)) = 0$.
  \end{enumerate}
\end{lemma}
\begin{proof}
  Since the inclusion $\Pol(\GG)\subset \Pol(\KK)$ must preserve the Haar state it extends to a trace preserving inclusion $L^\infty(\GG)\subset L^\infty(\KK)$ and by \cite[Theorem 6.29]{Luck02} the functor $L^\infty(\KK)\tens_{L^\infty(\GG)}-$ is therefore exact and dimension preserving.  By \cite[Proposition 2.1(iv) \& Theorem 3.11(v)]{reich01} the same is true for the functor $M(\KK)\tens_{L^\infty(\KK)}-$ and thus also for the composition $M(\KK)\tens_{L^\infty(\GG)}-$ of the two; hence
\[
\beta_p^{(2)}(\hat{\GG})=\dim_{L^\infty(\KK)} M(\KK)\tens_{L^\infty(\GG)} \Tor_p^{\Pol(\GG)}(L^\infty(\GG),\CC)  =\dim_{L^\infty(\KK)}\Tor_p^{\Pol(\GG)}(M(\KK),\CC)
\text{.}
\]
Moreover, by \cite[Corollary 3.4]{Thom06a}, dualizing is dimension preserving and thus  
\[
\dim_{L^\infty(\KK)}\Tor_p^{\Pol(\GG)}(M(\KK),\CC)=\dim_{L^\infty(\KK)}\Hom_{M(\KK)}\left(\Tor_p^{\Pol(\GG)}\big(M(\KK),\CC\big),M(\KK)\right)
\text{,}
\]
where the dimension on the right hand side is computed relative to the natural right $L^\infty(\KK)$-action on the dual module. Lastly, since $M(\KK)$ is a self-injective ring we get  (cf.~\cite[Theorem 3.5]{Thom06a} and its proof) an isomorphism of right $L^\infty(\KK)$-modules
\[
\Hom_{M(\KK)}\left(\Tor_p^{\Pol(\GG)}\big(M(\KK),\CC\big),M(\KK)\right)\simeq \Ext^p_{\Pol(\GG)}(\CC,M(\KK))
\text{,}
\]
and  upon computing the latter via the bar-resolution of the trivial $\Pol(\GG)$-module $\CC$, we see that $\Ext^1_{\Pol(\GG)}(\CC,M(\KK))$ identifies with $\Cohom^1(\Pol(\GG), M(\KK))$ as a right $L^\infty(\KK)$-module. In total we therefore obtain that
\[
\beta_1^{(2)}(\hat{\GG})=\dim_{L^\infty(\KK)}\Cohom^1(\Pol(\GG), M(\KK))
\text{.}
\]
This proves the first statement of the lemma.  For the second statement, 
note that by  \cite[Corollaries 3.3 and 3.4]{Thom06a} one has $\dim_{L^\infty(\KK)} \Cohom^1(\Pol(\GG), M(\KK)) = 0$ if and only if $ \Cohom^1(\Pol(\GG), M(\KK))=0$,  since $\Cohom^1(\Pol(\GG), M(\KK))$ can be seen as a dual module by what was just proven.
\end{proof}

Lemma \ref{lem:ltwo-betti-numbers-and-cohomology} will be used subsequently  to compute the first $\L^2$-Betti numbers of $\hat{\AA}_n ,\hat{O}_n^+, \hat{U}_n^+$ and  $\hat{\HH}_n$ as the dimension of their first cohomology groups with coefficients in $M(\HH_n)$.

\subsection{Cocycles on free products}
\label{sec:cocycles-free-products}

The primary aim in this section is to show that the free product construction in the category of complex algebras satisfies a universal property with respect to cocycles. However, this is most naturally done in the general setting of derivations into bimodules, but at the end of the section we will specialize to the case of polynomial algebras on compact quantum groups and their cocycles. The results stated are almost certainly well known to  the experts in the field, but since we were unable to find a suitable reference we have included this short account on the matter. In what follows,  let $A$ be a unital (complex) algebra and $\X$ an $A$-bimodule. 

\begin{lem}\label{derivation-lem}
The set $A\times \X$ is an algebra, with unit $(1,0)$, when endowed with the product $(a,x)\cdot(b,y):=(ab, ay+xb)$. Moreover,  if $\delta\colon A\to \X$ is a derivation then $\varphi\colon A\to A\times \X$ given by $\varphi(a)=(a,\delta(a))$ is a unital algebra-homomorphism
\end{lem}
\begin{proof}
This is all seen by straight forward calculations.
 \end{proof}

Denote by $A\ast B$ the free product \cite{vdn} of two unital algebras,  $A$ and $B$, and assume that $\X$ is an $A\ast B$-bimodule. It is therefore also an $A$-bimodule as well as a $B$-bimodule via the natural inclusions of the two algebras into $A\ast B$. Let now $\delta_1\colon A\to \X$ and $\delta_2\colon B\to \X$ be derivations and denote by $\varphi_1\colon A\to A\times \X\subset (A\ast B)\times \X$ and $\varphi_2\colon B\to (A\ast B)\times \X$ the corresponding algebra-homomorphisms given by Lemma \ref{derivation-lem}. Then, by the universal property of the free product, we obtain an algebra homomorphism
$\varphi=\varphi_1\ast \varphi_2\colon A\ast B\to (A\ast B)\times \X$ which restricts to the $\varphi_1$ and $\varphi_2$ respectively. Write the two components of $\varphi$ as $(\alpha,\delta)$; i.e.~$\alpha$ and $\delta$ arise, respectively, as $\varphi$ composed with the natural projections $(A\ast B)\times \X\to  A\ast B $ and $(A\ast B)\times \X\to \X$.
\begin{lem}\label{free-product-of-derivation-lem}
In the notation just introduced, the map $\alpha\colon A\ast B\to A\ast B$ is the identity map and the map $\delta\colon A\ast B\to \X$ is a derivation which extends $\delta_1$ and $\delta_2$, and is unique with this property.
\end{lem}

\begin{proof}
Using that $\varphi$ is an algebra-homomorphism we obtain, for $w_1,w_2\in A\ast B$, that
\begin{align} 
(\alpha(w_1w_2),\delta(w_1w_2)) &= \varphi(w_1w_2)=\varphi(w_1)\varphi(w_2) \notag\\
&=(\alpha(w_1), \delta(w_1))\cdot (\alpha(w_2), \delta(w_2))\notag\\
&=(\alpha(w_1)\alpha(w_2), \alpha(w_1)\delta(w_2) + \delta(w_1)\alpha(w_2)) \label{dagger}
\end{align}
It follows that $\alpha$ is multiplicative, and since $\varphi$ is a unital algebra-homomorphism, $\alpha$ is also unital and linear; that is, a unital algebra-homomorphism. Moreover, the construction of $\varphi_1$ and $\varphi_2$ implies that $\alpha$ restricts to the identity on both $A$ and $B$ and hence $\alpha$ is the identity map. Knowing this, the fact that $\delta$ is a derivation follows from the computation \eqref{dagger}. Moreover, it follows from the construction that the derivation $\delta$ restricts to the original derivations $\delta_1$ and $\delta_2$ respectively. Uniqueness follows directly from the derivation property by applying $\delta$ to words in elements from $A$ and $B$.
\end{proof}

Returning to the case of compact quantum groups, consider two such, $\GG$ and $\HH$, as well as a  left module $\X$ for the Hopf $*$-algebra $\Pol(\GG)\ast\Pol(\HH)$ \cite{wang}. Endowing $\X$ with the right action given by the counit, a 1-cocycle can be viewed as derivation, and Lemma \ref{free-product-of-derivation-lem} therefore shows that for any two 1-cocycles $c_1\colon \Pol(\GG) \to \X$ and $c_2\colon \Pol(\HH)\to \X $ there exists a unique 1-cocycle $c:=c_1\ast c_2 \colon \Pol(\GG)\ast \Pol(\HH) \to \X$ which extends $c_1$ and $c_2$.

\subsection{Co-amenability}
For the proofs in Section \ref{new-computations-section}, we will need the following characterization of (co-)amenability. Recall that a compact quantum group is \emph{co-amenable} if the counit $\varps\colon \Pol(\GG)\to \CC$ extends to $C(\GG)_\red$. This is equivalent to \emph{amenability} of the discrete dual $\hat{\GG}$; the latter being defined in terms of  the existence of an invariant state on $\ell^\infty(\hat{\GG})$ (cf.~\cite{murphy-tuset, tomatsu-amenable} for details). Regarding the  quantum groups under consideration here, it is known that $\hat{U}_n^+$ is non-amenable for all $n\geq 2$ and that $\hat{\AA}_n$ and $\hat{O}_n^+$ are amenable if and only if $n=2$ \cite{banica-sym, banica-libre-U_n}.

\begin{lem}\label{coamenable-lem}
A compact quantum group  is coamenable if and only if no element in $\ker(\varps)$ is invertible in $C(\GG)_\red$.
\end{lem}
We remark that the lemma follows directly from Kesten's criterion \cite[Theorem 6.1]{banica-subfactor}, but since it is easy to give a short and direct proof we have included it here for the benefit of the reader.
\begin{proof}
If $\GG$ is coamenable then $\varps$ extends to a character on $C(\GG)_\red$ and hence no element in its kernel can be invertible. Conversely, if $\GG$ is not coamenable, then $J:=\overline{\ker(\varps)}$ is a closed two-sided ideal in $C(\GG)_{\red}$ and we now claim that $J=C(\GG)_\red$. If this were not the case, then $1\notin J$ and hence we may extend $\varps$ to $J\oplus \CC 1$ by setting $\varps(x+\alpha 1)=\alpha$. However, $J\oplus \CC 1$ is easily seen to be closed 
and contains the dense subset $\Pol(\GG)$ so $J + \CC 1=C(\GG)_\red$, contradicting the fact that $\varps$ does not extend to $C(\GG)_\red$.
Thus,  $\ker(\varps)$ is dense in $C(\GG)_\red$ and since the group of invertible elements in $C(\GG)_\red$ is open it must intersect $\ker(\varps)$.
\end{proof}

\section{New computations of $\L^2$-Betti numbers}
\label{new-computations-section}

In this section we turn to the concrete computations of $\L^2$-Betti numbers announced  in the introduction.   In Section \ref{sec:vanishing-A-n} we show that all $\L^2$-Betti numbers of $\hat{\AA}_n$ vanish, thus proving Theorem~\ref{thm-B}.  We then proceed to the first $\L^2$-Betti number of the free product quantum groups $\hat \HH_n$ in Section \ref{sec:computation-H_n}.  Finally, in Section \ref{sec:computation-U_n} we combine these results to prove our main Theorem \ref{thm-A}. 
 Let us describe our strategy of proof for Theorem \ref{thm-A} in more detail.  We consider $\Pol(U_n^+) \subset \Pol(\HH_n)$ as described in Section \ref{sec:universal-quantum-groups} and may --- thanks to Lemma \ref{lem:ltwo-betti-numbers-and-cohomology} --- compute the first $\L^2$-Betti number of $\hat U_n^+$ as the von Neumann dimension of the $L^\infty(\HH_n)$-module $\Cohom^1(\Pol(U_n^+), M(\HH_n))$.  Since the first $\L^2$-Betti number of $\hat O_n^+$ vanishes, natural candidates for 1-cocycles representing cohomology are the restrictions to $\Pol(U_n^+)$ of free product 1-cocycles of the form $0 * c$ on $\Pol(\HH_n) = \Pol(O_n^+) * \Pol(S^1)$.  On the one hand, we analyse the behaviour of $0*c$ on $\Pol(\AA_n) \subset \Pol(U_n^+)$ in order to show that triviality of $0*c \hspace{-0.1cm}\restriction_{\Pol(U_n^+)}$ in $\Cohom^1(\Pol(U_n^+), M(\HH_n))$ implies $c = 0$.  On the other hand, we use vanishing of the first $\L^2$-Betti number of $\hat{\AA}_n$, together with Lemma \ref{lem:ltwo-betti-numbers-and-cohomology}, to show that every 1-cocycle on $\Pol(U_n^+)$ with values in $M(\HH_n)$ can be extended to $\Pol(\HH_n)$.  Vanishing of the $\L^2$-Betti numbers of $\hat O_n^+$ and another application of Lemma \ref{lem:ltwo-betti-numbers-and-cohomology} show that the extension can be represented by a free product 1-cocycle.  This establishes a one-to-one correspondence between elements in the 1-cohomology  $\Cohom^1(\Pol(U_n^+), M(\HH_n))$ and choices of the value $c(\id_{S^1}) \in M(\HH_n)$.  So, in turn, the first $\L^2$-Betti number of $\hat{U}_n^+$ equals the $L^\infty(\HH_n)$-dimension of $M(\HH_n)$, which is one.

\subsection{Vanishing of the $\L^2$-Betti numbers of $\hat{\AA}_n$}
\label{sec:vanishing-A-n}

We will prove Theorem \ref{thm-B}, by relating the $\L^2$-Betti numbers of $\hat \AA_n$ to those of $\hat O_n^+$, which are known to vanish \cite{thom-collins}.  The universal property of $\Pol(O_n^+)$ provides us with a unique $*$-automorphism $\alpha \colon \Pol(O_n^+) \to \Pol(O_n^+)$ satisfying $\alpha(v_{ij}) = -v_{ij}$ for all $i, j \in \{1, \dotsc, n\}$, and the induced $\ZZ/2\ZZ$-action on $\Pol(O_n^+)$ provides the necessary link between the two quantum groups.  For a left (respectively right) $\Pol(O_n^+)$-module $\X$  we denote by ${_\alpha}\X$ (respectively $\X_\alpha)$ the vector space $\X$ considered as a $\Pol(O_n^+)$-module via $\alpha$; i.e.~with left (respectively right) action given by $a.x:=\alpha(a)x$ (respectively $x.a=x\alpha(a))$.
The strategy of proof for Theorem \ref{thm-B} can now be described as follows. The $p$-th $\ell^2$-Betti number  $\beta_p^{(2)}(\hat{\AA}_n)$ is, by definition, the von Neumann dimension of the $L^\infty(\AA_n)$-module $\Tor_p^{\Pol(\AA_n)}(L^\infty(\AA_n),\CC)$.  We will reason that inducing this Tor-module to $L^\infty(O_n^+)$, we can calculate the von Neumann dimension of the $L^\infty(O_n^+)$-module $\Tor_p^{\Pol(\AA_n)}(L^\infty(O_n^+),\CC)$ instead.  A flat base change will then reduce our problem to a concrete identification of the $\Pol(O_n^+)$-module $\Pol(O_n^+) \tens_{\Pol(\AA_n)} \CC$, which turns out to split as a direct sum of the trivial module $\CC$ and the twisted trivial module ${_\alpha}\CC$.  Since twisting is compatible with Tor and preserves the von Neumann dimension, we will be able to conclude the proof by appealing to the known vanishing results for $\ell^2$-Betti numbers of $\hat O_n^+$.

\begin{lemma}
  \label{lem:C-plus-C}
  The left $\Pol(O_n^+)$-module $\Pol(O_n^+)\tens_{\Pol(\AA_n)}\CC$ is isomorphic to the direct sum $\CC \oplus {_\alpha}\CC$ where $\CC$ is considered a $\Pol(O_n^+)$-module via the counit.
\end{lemma}
\begin{proof}
  Denote by $1,z \in \CC[ \ZZ/2 \ZZ]$ the canonical unitaries.  Since $z$ is a self-adjoint unitary, the universal property of $\Pol(O_n^+)$ provides us with $*$-homomorphism $\pi \colon \Pol(O_n^+) \to \CC[\ZZ/2\ZZ]$ satisfying $\pi(v_{ij})=\delta_{ij}z$.  We consider $\CC[\ZZ/2\ZZ]$ as a left $\Pol(O_n^+)$-module via $\pi$.  Note that $\CC[\ZZ/2\ZZ] \cong \CC \oplus {_\alpha}\CC$ as left $\Pol(O_n^+)$-modules, since on both modules the $\Pol(O_n^+)$-action factors through $\pi$ and defines the left-regular representation of $\ZZ/2\ZZ$.  In order to prove the lemma, we will show that $\Pol(O_n^+)\tens_{\Pol(\AA_n)}\CC \cong \CC[\ZZ/2\ZZ]$.  
 Since 
  \begin{equation*}
    \pi(v_{ij}v_{kl}) = \delta_{i,j} \delta_{k,l} 1 = \varps(v_{ij}v_{kl})1
  \end{equation*}
 for all $i,j,k,l \in \{1, \dotsc, n\}$, we obtain a $\Pol(O_n^+)$-modular map 
 \begin{equation*}
   \pi \tens_{\Pol(\AA_n)} \id \colon \Pol(O_n^+) \tens_{\Pol(\AA_n)} \CC \to \CC[\ZZ/2\ZZ]
   \text{,}
 \end{equation*}
 which is obviously surjective.  We prove injectivity of $\pi \tens_{\Pol(\AA_n)} \id$, by showing that the dimension of $\Pol(O_n^+) \tens_{\Pol(\AA_n)} \CC$ is at most two.  Let us write $[x] = x \tens 1$ for the image of $x \in \Pol(O_n^+)$ in $\Pol(O_n^+) \tens_{\Pol(\AA_n)} \CC$.  Since $v_{ij} v_{kl} \in \Pol(\AA_n)$, it is clear that $\Pol(O_n^+) \tens_{\Pol(\AA_n)} \CC$ is spanned by the elements $[1]$ and $[v_{ij}]$ for $i,j \in \{ 1, \dotsc, n \}$.  Furthermore, for $i\neq j$ we have 
\begin{equation*}
  [v_{ij}]
  =
  \left(\sum_{k=1}^n v_{k1}v_{k1} \right)
  [v_{ij}]=\sum_{k=1}^n v_{k1}[v_{k1}v_{ij}]=\sum_{k=1}^n v_{k1}[\varps(v_{k1})\varps(v_{ij}))1]
  =
  0
  \text{,}
\end{equation*}
and
\begin{equation*}
  [v_{ii}]
  =
  \left( \sum_{k=1}^n v_{jk}v_{jk} \right)[v_{ii}]
  =
  \sum_{k=1}^n v_{jk}[v_{jk}v_{ii}]
  =
  \sum_{k=1}^n v_{jk}[\varps(v_{jk})\varps(v_{ii})1]
  =
  v_{jj}[1]
  =
  [v_{jj}]
  \text{.}
\end{equation*}
So $\Pol(O_n^+) \tens_{\Pol(\AA_n)} \CC = \spann\{ [1], [v_{11}]\}$.  This finishes the proof of the lemma.
\end{proof}

Before stating the next lemma, we remark that $\alpha$ extends to a trace preserving $*$-auto\-mor\-phism of $L^\infty(O_n^+)$.  Indeed, \cite[Theorem 4.1]{banica-collins1} says that the Haar state $h$ of $\Pol(O_n^+)$ vanishes on words of odd length in the generators $v_{ij}$, from which we deduce that $h \circ \alpha = h$ and hence that $\alpha$ extends as claimed.   We will apply the notation ${_\alpha}\X$ and $\X_\alpha$ to left- and right $L^\infty(O_n^+)$-modules as we did for $\Pol(O_n^+)$-modules before.
\begin{lemma}
  \label{lem:vanishing-dimension}
  We have
    $
    \dim_{L^\infty(O_n^+)}\Tor_p^{\Pol(O_n^+)}\left(L^\infty(O_n^+),{_\alpha}\CC\right)
    =
    0
    \text{.}
    $
\end{lemma}
\begin{proof}
  We first perform a flat base change \cite[Proposition 3.2.9]{weibel} via $\alpha$ to obtain an isomorphism of left $L^\infty(O_n^+)$-modules
  \begin{equation*}
      \Tor_p^{\Pol(O_n^+)}(L^\infty(O_n^+), {_\alpha}\CC)
    \simeq
    \Tor_p^{\Pol(O_n^+)}\left( L^\infty(O_n^{+}) \tens_{\Pol(O_n^{+})} {_\alpha}\Pol(O_n^+) , \CC  \right).
  \end{equation*}
  Since $\alpha$ is self-inverse, the map
  \begin{equation*}
    L^\infty(O_n^{+}) \tens_{\Pol(O_n^{+})} {_\alpha}\Pol(O_n^+)
    \ni
    x\tens a
    \mapsto
    \alpha(x)a\in {_\alpha}L^\infty(O_n^+)
  \end{equation*}
  is an isomorphism of left $L^\infty(O_n^+)$-modules and hence
  \begin{align*}
    \dim_{L^\infty(O_n^+)} \Tor_p^{\Pol(O_n^+)}(L^\infty(O_n^+), {_\alpha} \CC)
    & =
      \dim_{L^\infty(O_n^+)} \Tor_p^{\Pol(O_n^+)}\left( {_\alpha}L^\infty(O_n^{+}) , \CC  \right) \\
    & =
      \dim_{L^\infty(O_n^+)} {_\alpha} \hspace{-0.1cm} \left(\Tor_p^{\Pol(O_n^+)}\left( L^\infty(O_n^{+}) , \CC  \right)\right)
      \text{.}
  \end{align*}
  Note that the endo-functor $\X \mapsto {_\alpha}\X$ on the category of left $L^\infty(O_n^+)$-modules maps the class of finitely generated projective modules onto itself, is dimension-preserving on this class and preserves inclusions; hence $\dim_{L^\infty(O_n^+)}(\X) = \dim_{L^\infty(O_n^+)}({_\alpha}\X)$ for all $L^\infty(O_n^+)$-modules $\X$ (cf.~\cite[Section 6.1]{Luck02}). Combined with the previous calculation, this gives
  \begin{align*}
    \dim_{L^\infty(O_n^+)}\Tor_p^{\Pol(O_n^+)}(L^\infty(O_n^+), {_\alpha}\CC)
     =
      \dim_{L^\infty(O_n^+)}\Tor_p^{\Pol(O_n^+)}(L^\infty(O_n^+), \CC) 
     =
      \beta_p^{(2)}(\hat{O}_n^+) 
     =
      0
      \text{,}
  \end{align*}
  where we used the vanishing of the $\L^2$-Betti numbers of $\hat{O}_n^+$ from \cite{thom-collins}.
\end{proof} 

We are now ready to prove that the $\L^2$-Betti numbers of $\hat \AA_n$ vanish.
\begin{proof}[Proof of Theorem \ref{thm-B}]
  We first notice that
  \begin{align*}
    \beta_p^{(2)}(\hat{\AA}_n)
    & \stackrel{\text{def}}{=}
      \dim_{L^\infty(\AA_n)}\Tor_p^{\Pol(\AA_n)}(L^\infty(\AA_n),\CC) \notag \\
    & =
      \dim_{L^\infty(O_n^+)}L^\infty(O_n^+)\tens_{L^\infty(\AA_n)}\Tor_p^{\Pol(\AA_n)}(L^\infty(\AA_n),\CC) \notag \\
    & =
      \dim_{L^\infty(O_n^+)}\Tor_p^{\Pol(\AA_n)}(L^\infty(O_n^+),\CC) \notag \\
    & =
      \dim_{L^\infty(O_n^+)}\Tor_p^{\Pol(O_n^+)}\left(L^\infty(O_n^+),\Pol(O_n^+)\tens_{\Pol(\AA_n)}\CC\right)
      \text{.}
  \end{align*}
  Here the first step follows since the functor $L^\infty(O_n^+)\tens_{L^\infty(\AA_n)} -$ is dimension preserving {\cite[Theorem 6.29 (2)]{Luck02}} and the second step follows since it is exact {\cite[Theorem 6.29 (1)]{Luck02}} and therefore commutes with $\Tor$. 
  The last step follows by applying the flat base change formula (cf.~\cite[Proposition 3.2.9]{weibel}) to the inclusion $\Pol(\AA_n)\subset \Pol(O_n^+)$ which was proven to be (faithfully) flat in \cite{chirvasitu-faithfully-flat}. 
Evoking Lemma \ref{lem:C-plus-C}, we therefore obtain
\begin{align*}
  \beta_p^{(2)}(\hat{\AA}_n)
  =
  \dim_{L^\infty(O_n^+)}\Tor_p^{\Pol(O_n^+)}\left(L^\infty(O_n^+),\CC\right) +
  \dim_{L^\infty(O_n^+)}\Tor_p^{\Pol(O_n^+)}\left(L^\infty(O_n^+),{_\alpha}\CC\right) 
\end{align*}
The first term is the $p$-th $\L^2$-Betti number of $\hat{O}_n^+$, which vanishes by \cite{thom-collins}, and the second term vanishes by Lemma \ref{lem:vanishing-dimension}.  We conclude that $\beta_p^{(2)}(\hat{\AA}_n) = 0$, finishing the proof of the theorem.
\end{proof}

\subsection{The first $\L^2$-Betti number of $\hat \HH_n$}
\label{sec:computation-H_n}

In addition to the results in Section \ref{sec:vanishing-A-n}, the second important ingredient in the proof of Theorem \ref{thm-A} is finding representatives for the classes in 1-cohomology of $\Pol(\HH_n)$ with values in $M(\HH_n)$.  On our  way, we calculate $\beta_1^{(2)}(\hat \HH_n)$ to be 1.

\begin{lemma}
  \label{dim-of-Z}
  The right $L^\infty(\HH_n)$-modules $M(\HH_n)$ and 
  \begin{equation*}
    {\Z}:=\{c\in Z^1(\Pol(\HH_n), M(\HH_n)) \mid \forall i,j: \, c(v_{ij})=0\}
  \end{equation*}
  are isomorphic; in particular the dimension of the latter is 1.
\end{lemma}
\begin{proof}
 Each $\xi \in M(\HH_n)$ defines an inner cocycle $c_\xi\colon \Pol(S^1)\to M(\HH_n)$ given by $c_\xi(x)=x.\xi-\varps(x)\xi$ and we therefore obtain a map $\varphi\colon M(\HH_n) \to \Z$ given by $\varphi(\xi):=c_{\xi}\ast 0$, where $0$ denotes the zero-cocycle on $\Pol(O_n^+)$. A direct verification shows that $\varphi$ is a morphism of right $L^\infty(\HH_n)$-modules and we now prove that it is injective. If $\varphi(\xi)=0$ then $0=z.\xi-\xi=(z-1).\xi$. However, $z-1$ is invertible in $M(S^1)$, and hence also in the over-ring $M(\HH_n)$, so $\xi=0$,
showing that $\varphi$ is indeed an embedding. On the other hand, since $0=\beta_1^{(2)}(\ZZ)=\dim_{L^\infty(\HH_n)}\Cohom^1(\Pol(S^1),M(\HH_n))$, every cocycle $c\colon \Pol(S^1)\to M(\HH_n)$ is inner (cf.~Lemma \ref{lem:ltwo-betti-numbers-and-cohomology}), and for $c\in \Z$ we therefore obtain a $\xi\in M(\HH)$  which implements $c$ on the subalgebra $\Pol(S^1)$. It is clear that $\varphi(\xi)$ and $c$ agree on the entries  $\{z, v_{ij}\mid i,j=1,\dots, n\}$ of the fundamental unitary corepresentation $v\oplus z$ and by Lemma  \ref{uniquely-determined-lem} we conclude that $\varphi(\xi)=c$, thus proving that $\varphi$ is surjective.
\end{proof}

\begin{prop}
  \label{beta-one-of-HHn}
  For any $n\geq 2$ we have $\beta_1^{(2)}{(\hat{\HH}_n)}=1$.
\end{prop}
\begin{proof}
  
  By Lemma \ref{dim-of-Z}, it suffices to show that the natural quotient map from 1-cocycles to 1-cohomology induces an isomorphism $ \kappa\colon \Z \to \Cohom^1(\Pol(\HH_n), M(\HH_n))$ of $L^\infty(\HH_n)$-modules.  Since $\beta_1^{(2)}(\hat{O}_n^+)=0$ \cite{vergnioux-paths-in-cayley}, Lemma \ref{lem:ltwo-betti-numbers-and-cohomology} says that any cocycle $c\colon \Pol(\HH_n) \to M(\HH_n)$ is equivalent to a cocycle vanishing on $\Pol(O_n^+)$ and hence $\kappa$ is surjective. To prove injectivity, assume that $c\in \Z$ is inner, say, implemented by a vector $\zeta\in M(\HH_n)$ which then  satisfies  $x\zeta=\varps(x)\zeta$ for all $x\in \Pol(O_n^+)$ since $c$ is assumed to vanish on $\Pol(O_n^+)$. If $n\geq 3$, the discrete quantum group $\hat{O}_n^+$ is non-amenable, so by  Lemma \ref{coamenable-lem} there exists $y_0\in \Pol(O_n^+)\cap \ker(\varps)$ which is invertible as an operator in $C(O_n^+)_\red$. In particular, $y_0$ is invertible in $M(O_n^+)$, thus also in the over-ring $M(\HH_n)$, and the relation $y_0\zeta=\varps(y_0)\zeta=0$ therefore forces $\zeta=0$; whence $\kappa$ is injective. If $n=2$, then $\hat{O}_2^+$ is amenable  and  $\Pol(O_2^+)$ is a domain (cf. Chapter I.1 in \cite{brown-goodearl} or the Appendix). By \cite[Theorem 3.4]{applications-of-foelner}, this implies the existence of a skew field  between $\Pol(O_2^+)$ and $M(O_2^+)$ and therefore \emph{any} non-zero element in $\Pol(O_2^+)$ is invertible in $M(O_2^+)$, and hence also in the over-ring $M(\HH_2)$. Since $x\zeta=\varps(x)\zeta$ for any $x\in \Pol(O_2^+)$, by choosing $x$ as a non-zero element in $\Pol(O_2^+)\cap \ker(\varps)$ we conclude again that $\zeta=0$.
\end{proof}
\begin{rem}
The higher cohomology of  free products is well understood, and in the case of $\Pol(\HH_n)$ one has
\[
\Cohom^p(\Pol(\HH_n), M(\HH_n)) \simeq \Cohom^p(\Pol(O_n^+), M(\HH_n)) \oplus \Cohom^p(\Pol(S^1),M(\HH_n)), \quad p\geq 2.
\]
For a proof, cf.~\cite{bichon-cohom-dim}. Since the $\ell^2$-Betti numbers of $\hat{O}_n^+$ and $\ZZ$ vanish in degrees higher than $1$ this implies that $\beta_p^{(2)}(\hat{\HH}_n)=0$ for $p\geq 2$. Note also that $\beta_0^{(2)}(\hat{\HH}_n)=0$ by \cite{beta-zero}.
\end{rem}

\subsection{The first $\L^2$-Betti number of $\hat{U}_n^+$}
\label{sec:computation-U_n}

We are now ready to prove our main Theorem \ref{thm-A}.  Let us fix the following short calculation for later use.
\begin{lemma}
  \label{lem:relate-cocycles}
  Let $\X$ be a left $\Pol(\HH_n)$-module and let $c \in Z^1(\Pol(\HH_n), \X)$ be a 1-cocycle that is trivial on $\Pol(O_n^+)$.  Then $c(u_{ij}) = \delta_{ij} c(z)$ for all $i,j \in \{1, \dotsc, n\}$.
  In particular, if $c(v_{ij}) = c(u_{ij}) = 0$ for all $i,j \in \{1, \dotsc, n\}$ then $c = 0$.
\end{lemma}
\begin{proof}
  Let $c \in Z^1(\Pol(\HH_n), \X)$ be trivial on $\Pol(O_n^+)$.  For all $i,j \in \{1, \dotsc, n\}$ we then have
  \begin{equation*}
    c(u_{ij})
    =
    c(z v_{ij})
    =
    z c(v_{ij}) + c(z) \varps(v_{ij})
    =
    \delta_{ij} c(z)
    \text{.}
  \end{equation*}
  If we now assume that $c(v_{ij}) = c(u_{ij}) = 0$ for all $i,j \in \{1, \dotsc, n\}$, then $c$ vanishes on $\Pol(O_n^+)$ by Lemma \ref{uniquely-determined-lem}.  Further, we obtain $c(z) = c(u_{11}) = 0$ and hence another application of Lemma \ref{uniquely-determined-lem} shows that $c = 0$.  
\end{proof}

The next lemma provides  an extension result making it possible to compare the first $\L^2$-Betti number of $\hat U_n^+$ with that of $\hat \HH_n$.
\begin{lem}
  \label{extension-lem}
  Let $c\in Z^1(\Pol(U_n^+), M(\HH_n))$.  If there exists some $\xi \in M(\HH_n)$ such that for all $i,j \in \{1, \dotsc, n\}$ we have $c(u_{ij})=\varps(u_{ij})\xi$, then there exists $\tilde{c} \in Z^1(\Pol(\HH_n), M(\HH_n))$ that extends $c$.
\end{lem}
\begin{proof}
  Assume that $c(u_{ij}) = \varps(u_{ij})\xi$ for all $i, j \in \{1, \dotsc, n\}$.  Since $\ZZ$ is free, the formula $c_1(z) := \xi$ defines a unique cocycle on $\ZZ$ with values in $M(\HH_n)$, which extends to a cocycle \mbox{$c_1 \colon \Pol(S^1) = \CC[\ZZ] \to M(\HH_n)$} by linearity.  Denote by  $0$ the zero-cocycle on $\Pol(O_n^+)$, and consider  the free product cocycle $\tilde{c} := c_1 \ast 0$ as explained in Section \ref{sec:cocycles-free-products}.  Lemma \ref{lem:relate-cocycles} shows that
  \begin{equation*}
    \tilde{c}(u_{ij})
    =
    \delta_{ij} \tilde{c}(z)
    =
    \varps(u_{ij}) \xi
    =
    c(u_{ij})
    \text{,}
  \end{equation*}
  for all $i,j \in \{1, \dotsc, n\}$.  By Lemma \ref{uniquely-determined-lem},  a cocycle on $\Pol(U_n^+)$ is uniquely determined on the matrix coefficients $u_{ij}$, $i,j \in \{1,\dots, n\}$, so  $\tilde c$ is indeed an extension of $c$,  and the proof of the lemma is complete.
\end{proof}

We now turn to the proof of our main result.
\begin{proof}[Proof of Theorem \ref{thm-A}]
  Applying Lemma \ref{lem:ltwo-betti-numbers-and-cohomology} to the inclusion $\Pol(U_n^+) \subset \Pol(\HH_n)$, we obtain
  \begin{equation*}
    \beta_1^{(2)} (\hat{U}_n^{+})
    =
    \dim_{L^\infty(\HH_n)} \Cohom^1(\Pol(U_n^+), M(\HH_n))
    \text{.}
  \end{equation*}
  Consider the set
  \begin{equation*}
    \Z
    :=
    \{ c \in Z^1(\Pol(\HH_n), M(\HH_n)) \mid c(v_{ij})=0 \}
    \text{,}
  \end{equation*}
  and the composition
  \begin{equation*}
    \alpha:
    \Z
    \stackrel{\iota}{\hookrightarrow}
    Z^1(\Pol(\HH_n), M(\HH_n))
    \stackrel{\mathrm{res}}{\longrightarrow}
    Z^1(\Pol(U_n^+), M(\HH_n))
    \stackrel{\pi}{\twoheadrightarrow}
    \Cohom^1(\Pol(U_n^+), M(\HH_n))
    \text{.}
  \end{equation*}
  We show that $\alpha$ is an isomorphism of $L^\infty(\HH_n)$-modules, which proves Theorem \ref{thm-A} thanks to Lemma \ref{dim-of-Z}.
Let us first prove injectivity of $\alpha$.  Assuming that $\alpha(c) = 0$ for some $c \in \Z$ amounts to saying that $c$ is inner on $\Pol(U_n^+)$ --- say implemented by a vector $\zeta \in M(\HH_n)$.  Since $c(v_{ij}) = 0$ for all $i,j \in \{1, \dotsc, n\}$ and $v_{ij} v_{kl}=u_{ij}^* u_{kl}\in \Pol(U_n^+) $ for all  for all $i,j,k,l \in \{1, \dotsc, n\}$, we therefore get
\[  
0= c(v_{ij} v_{kl})= v_{ij} v_{kl}\zeta -\varps(v_{ij} v_{kl})\zeta, 
\]  
 and hence $a \zeta = \varps(a) \zeta$ for all $a\in \Pol(\AA_n)$. \\

\noindent  \textit{Case 1}.  If $n=2$, then $\hat{\AA}_2$ is amenable by \cite[Corollary 4.1 \& 4.2]{banica-sym}.  Further, since $\Pol(\AA_2) \subset \Pol(O_2^+)$ and the latter is known to be a domain (see Chapter I.1 of \cite{brown-goodearl} or the Appendix),  $\Pol(\AA_2)$ is also a domain, and   by \cite[Theorem 3.4]{applications-of-foelner}, this implies the existence of a skew field between $\Pol(\AA_2)$ and $M(\AA_2)$. Hence any non-zero element in $\Pol(\AA_2)$ is invertible in $M(\AA_2)$ and thus in the over-ring $M(\HH_2)$ as well.  For every $a\in \Pol(\AA_2) \cap \ker(\varps)$ we have
  \begin{equation*}
    \zeta
    =
    a^{-1} a \zeta
    =
    a^{-1} \varps(a) \zeta
    =
    0
    \text{.}
  \end{equation*}
  This shows $c\hspace{-0.1cm}\restriction_{\Pol(U_n^+)} = 0$ and Lemma~\ref{lem:relate-cocycles} now finishes the proof of injectivity of $\alpha$ when $n = 2$. \\

\noindent  \textit{Case 2}.  If $n\geq 3$, then $\hat{\AA}_n$ is non-amenable, and by Lemma \ref{coamenable-lem}
  this means that there exists $x \in \Pol(\AA_n)\cap \ker(\varps)$ which is invertible as an operator in $C(\AA_n)_\red$.  This element is therefore also invertible in the bigger  $C^*$-algebra $C(\HH_n)_\red$ and thus in $M(\HH_n)$ as well.  The proof is now finished in the same way as in Case 1:  we obtain
  \begin{equation*}
    \zeta
    =
    x^{-1} x \zeta
    =
    x^{-1} \varps(x) \zeta
    =
    0
    \text{.}
  \end{equation*}
  We therefore have $c \hspace{-0.1cm}\restriction_{\Pol(U_n^+)} = 0$ and Lemma \ref{lem:relate-cocycles} finishes the proof of injectivity of $\alpha$ for $n \geq 3$.

  It remains to show that $\alpha$ is surjective.  Let $[c] \in \Cohom^1(\Pol(U_n^+), M(\HH_n))$ be given.  By Theorem \ref{thm-B}, we have $\beta_1^{(2)}(\hat{\AA}_n) = 0$, which  implies that $\Cohom^1(\Pol(\AA_n), M(\HH_n)) = 0$ by Lemma \ref{lem:ltwo-betti-numbers-and-cohomology}.  So we may assume that $c$ vanishes on $\Pol(\AA_n)$.  The formula \eqref{unitary-eq-2} now gives
  \begin{equation*}
    c(u_{ij})
    =
    \sum_{k=1}^n c(u_{1k}u_{1k}^* u_{ij})
    =
    \sum_{k=1}^n
    u_{1k}\underbrace{c(u_{1k}^*u_{ij})}_{=0} +
    c(u_{1k})\varps(u_{1k}^*u_{ij})
    =
    \varps(u_{ij}) c(u_{11})
    \text{,}
  \end{equation*}
and we may therefore apply Lemma \ref{extension-lem} and find an extension $\tilde{c} \colon \Pol(\HH_n) \to M(\HH_n)$ of $c$.  Since $\beta_1^{(2)}(\hat O_n^+) = 0$,  Lemma \ref{lem:ltwo-betti-numbers-and-cohomology} shows that $\tilde{c} \in Z^1(\Pol(\HH_n),M(\HH_n))$ is cohomologous to a cocycle $\tilde{c}'\in \Z$ which, by construction, satisfies $\alpha(\tilde{c}')=[c]$, thus showing surjectivity of $\alpha$ and finishing the proof of Theorem \ref{thm-A}.
\end{proof}

\begin{rem}
Since  $\Pol(O_n^+)$ and $\Pol(S^1)$ are of cohomological dimension  3 and 1, respectively,   \cite[Corollary 2.5 ]{bergman}  gives that  $\Pol(\HH_n)$ has cohomological dimension 3, and combining \cite[Theorem 2.1]{chirvasitu-faithfully-flat} with \cite[Corollar 1.8]{schneider} we furthermore have that $\Pol(\HH_n)$ is projective as a $\Pol(U_n^+)$-module. From this we deduce that $\Pol(U_n^+)$ has cohomological dimension at most 3, and equality follows since the subring $\Pol(\AA_n)$ is known to be of cohomological dimension 3 \cite[Theorem 6.5]{bichon-gerstenhaber}. This implies  that $\beta_p^{(2)}(\hat{U}_n^+)=0$ for $p\geq 4$. Note also that $\beta_0^{(2)}(\hat{U}_n^+)=0$ by \cite{beta-zero}.
\end{rem}

\begin{added}
The results in the present paper have subsequently been generalized by Julien Bichon and the authors in \cite{BKR16} to also include a computation of $\beta_2^{(2)}(\hat{U}_n^+)$ and $\beta_3^{(2)}(\hat{U}_n^+)$ which both turn out to be zero. For an even more general approach to these results,  the reader is referred to \cite[Theorem 5.2]{KRVV17} which furthermore contains a number of additional computations of $\ell^2$-Betti numbers for discrete quantum groups.

\end{added}

\appendix
\section{}\label{appendix-a}
In  this section we provide a proof of the fact that $\Pol(O_2^+)$ is a domain; i.e.~that it has no non-trivial zero-divisors.  This fact can be deduced from \cite[Chapter I.1]{brown-goodearl} using the well known identification of $\Pol(O_2^+)$ with $\Pol(SU_{-1}(2))$ (cf.~\cite[Proposition 5 \& 6]{banica-unitary}) and the fact that the underlying rings of $\Pol(SU_{-1}(2))$ and of $\Pol(SL_{-1}(2))$ are isomorphic.  For the benefit of the reader, we give a short proof in operator algebraic terminology, only using the identification $\Pol(O_2^+) \cong \Pol(SU_{-1}(2))$.

We denote by $\alpha$ and $\gamma$ the canonical generators of $\Pol(SU_{-1}(2))$, and recall \cite{woronowicz}   that the defining relations are
\begin{center}
\begin{tabular}{rl}
\begin{minipage}{5cm}
\begin{eqnarray*}
\alpha^*\alpha + \gamma^*\gamma &= &1\\  
\alpha\alpha^* + \gamma\gamma^*&= &1\\
\gamma\gamma^*-\gamma^*\gamma &= & 0
\end{eqnarray*}
\end{minipage}
&
\begin{minipage}{5cm}
\begin{eqnarray*}
\alpha\gamma + \gamma\alpha &= & 0\\
\alpha\gamma^*+ \gamma^*\alpha &= &0\\  
                 & 
\end{eqnarray*}
\end{minipage}
\end{tabular}
\end{center}
Note that this implies that $\alpha^*\alpha=\alpha\alpha^*$ and that $\gamma\gamma^*$ is a central element.
In the following we use the convention that $\alpha^{i}=(\alpha^*)^{-i}$ for $i<0$ and $x^0=1$ for $x\neq 0$. 
By \cite{woronowicz} the set
\[
\B:=\{\alpha^i \gamma^{j}\gamma^{*k} \mid i\in \ZZ,j,k\in \NN_0\}
\]
constitutes a linear basis for $\Pol(SU_{-1}(2))$. For a non-zero element $x=\sum \lambda_{ijk}\alpha^i \gamma^{j}\gamma^{*k} \in\Pol(SU_{-1}(2))$ we define its degrees with respect to the basis:
\begin{align*}
\deg_\alpha(x)&:=\max\{i\in \ZZ\mid \exists j,k\in\NN_0: \lambda_{ijk}\neq 0 \};\\
\deg_{\gamma,\gamma^*}(x) &:=\max\{p\in \NN \mid \exists i\in \ZZ, k,l\in \NN_0 : \lambda_{ijk}\neq 0 \ \text{and} \ p=j+k\}.
\end{align*}

\begin{prop}
The ring $\Pol(SU_{-1}(2))$ is a domain.
\end{prop}

\setcounter{claim-counter}{0}

\begin{proof}
We first prove the following claim: 

\begin{claim}\label{claim1}
For $i,j\in \ZZ$ there exists a polynomial $p_{i,j}\in \CC[X,Y]$ such that $\alpha^i\alpha^j= \alpha^{i+j}p_{i,j}(\gamma,\gamma^*)$
\end{claim}
\begin{proof}[Proof of Claim \ref{claim1}] When $i$ and $j$ have the same sign this is clear --- the constant polynomial 1 does the job. If $i\geq 0$, $j<0$
and $i>|j|$ then
\[
\alpha^i \alpha^j=\underbrace{\alpha\cdots\alpha}_{\text{$i$}}\underbrace{\alpha^*\cdots \alpha^*}_{\text{$-j$}}=\alpha^{i+j} (\alpha\alpha^*)^{-j}=\alpha^{i+j}(1-\gamma\gamma^*)^{-j},
\]
and, similarly, if $i\leq -j$ we get 
\[
\alpha^i \alpha^j=\underbrace{\alpha\cdots\alpha}_{\text{$i$}}\underbrace{\alpha^*\cdots \alpha^*}_{\text{$-j$}}=\alpha^{*(-i-j)} (\alpha\alpha^*)^{i}=\alpha^{i+j}(1-\gamma\gamma^*)^{i}.
\]
The remaining case ($i<0$ and $j\geq 0$) follows by symmetry.
\end{proof}

\begin{claim}\label{claim2}
The element $\alpha^i$ is not a left zero-divisor for any $i\in \ZZ$.
\end{claim}
\begin{proof}[Proof of Claim \ref{claim2}]
Since $\alpha^*\alpha=\alpha\alpha^*$, it suffices to prove that $\alpha\alpha^*$ (and hence none of its powers) is a left zero-divisor. To this end, assume that $x\in\Pol(SU_{-1}(2))$ satisfies  $\alpha\alpha^* x=0$. Since $\alpha\alpha^*=1-\gamma\gamma^*$, this means $x=\gamma\gamma^*x$ and by expanding $x$ as $x=\sum_{i,j,k}\lambda_{ijk}\alpha^i\gamma^j\gamma^{*k}$  and using that $\gamma\gamma^*$ is central this translates into
\[
\sum_{i,j,k}\lambda_{ijk}\alpha^i\gamma^j\gamma^{*k}=\sum_{i,j,k}\lambda_{ijk}\alpha^i\gamma^{j+1}\gamma^{*(k+1)}.
\]
If these terms were non-zero, we could apply $\deg_{\gamma, \gamma^*}(-)$ on both sides to obtain a contradiction.  Hence $x=0$. 
\end{proof}

We now turn to the actual proof of the fact that $\Pol(SU_{-1}(2))$ is a domain. Let $x=\sum_{ijk}\lambda_{ijk} \alpha^i \gamma^j\gamma^{*k}$ and $y=\sum_{lmn}\mu_{lmn}\alpha^l\gamma^m\gamma^{*n} $ in $\Pol(SU_{-1}(2))$ be non-zero elements and denote their $\alpha$-degrees by $i_0$ and $l_0$, respectively. Assuming that $xy=0$, we have
\[
0=\sum_{\substack{i,j,k\\l,m,n}}\lambda_{ijk} \mu_{lmn} \alpha^i \gamma^j\gamma^{*k}\alpha^l\gamma^m\gamma^{*n}
\]
and hence 
\begin{align}
  \label{degree-eq}
  \sum_{\substack{i \leq i_0, l \leq l_0 \\ i + l < i_0 + l_0 \\ j,k,m,n}}
  \lambda_{ijk} \mu_{lmn} \alpha^i \gamma^j\gamma^{*k}\alpha^l\gamma^m\gamma^{*n} 
  & =
  -\sum_{\substack{j,k\\ m,n}}
  \lambda_{i_0jk}\mu_{l_0mn} \alpha^{i_0} \gamma^j\gamma^{*k}\alpha^{l_0}\gamma^m\gamma^{*n} \notag \\
  & = -\sum_{\substack{j,k\\ m,n}}
  \lambda_{i_0jk} \mu_{l_0mn}(-1)^{(j+k)|l_0|} \alpha^{i_0}\alpha^{l_0} \gamma^j\gamma^{*k}\gamma^m\gamma^{*n}\notag \\
  & =
  -\sum_{\substack{j,k\\ m,n}}
  \lambda_{i_0jk} \mu_{l_0mn}(-1)^{(j+k)|l_0|} \alpha^{i_0}\alpha^{l_0} \gamma^{j+m}\gamma^{*(k+n)}
  \text{.}
\end{align}
Applying Claim \ref{claim1}, we see that  the last term has the form $\alpha^{i_0+l_0}p(\gamma,\gamma^*)$ for some $p\in \CC[X,Y]$ and hence its $\alpha$-degree is $i_0+l_0$ if $p(\gamma^*,\gamma)\neq 0$. Similarly, we get that 
\[
\deg_\alpha\left(\sum_{\substack{i<i_0,l<l_0\\ j,k,m,n}}\lambda_{ijk} \mu_{lmn} \alpha^i \gamma^j\gamma^{*k}\alpha^l\gamma^m\gamma^{*n}   \right)<i_0+l_0
\]
and hence the equality \eqref{degree-eq} can only happen if $p(\gamma,\gamma^*)=0$. We therefore have
\begin{align*}
0&=\sum_{\substack{j,k\\ m,n}}\mu_{l_0mn}\lambda_{i_0jk}(-1)^{(j+k)|l_0|} \alpha^{i_0}\alpha^{l_0} \gamma^{j+m}\gamma^{*(k+n)}\\
&= \alpha^{i_0}\alpha^{l_0}\left(\sum_{\substack{j,k\\ m,n}}\mu_{l_0mn}\lambda_{i_0jk}(-1)^{(j+k)|l_0|}\gamma^{j+m}\gamma^{*(k+n)}\right)
\end{align*}
and, by Claim \ref{claim2}, this implies that 
\begin{align*}
0&=\sum_{\substack{j,k\\ m,n}}\mu_{l_0mn}\lambda_{i_0jk}(-1)^{(j+k)|l_0|}  \gamma^{j+m}\gamma^{*(k+n)} \\
&=\left(\sum_{j,k} \lambda_{i_0,j,k} (-1)^{(j+k)|l_0|}\gamma^j\gamma^{*k}  \right)\left( \sum_{m,n} \mu_{l_0mn}  \gamma^m\gamma^{*n}\right)
\end{align*}
However, since $\B$ is a linear basis for $\Pol(SU_{-1}(2))$, the map $\CC[X,Y]\ni p\mapsto p(\gamma,\gamma^*)\in \Pol(SU_{-1}(2))$ is injective, and since $\CC[X,Y]$ is a domain one of the factors in the last product needs to be zero, which contradicts the fact that $i_0$ and $l_0$ are chosen such that there exist $j,k,l,m\in \NN_0$ with $\lambda_{i_0jk}\neq 0$ and $\mu_{l_0,n,m}\neq 0$.

\end{proof}


\def\cprime{$'$} \def\cprime{$'$}

\end{document}